\documentclass[12pt,oneside]{amsart}

\usepackage{amssymb, amsthm, hyperref}
\usepackage{enumerate, amsfonts, amsxtra,amsmath,latexsym,epsfig, color}
\usepackage{mathrsfs, MnSymbol}
\usepackage{geometry}                		
\usepackage{graphicx}
\usepackage{hyperref}				
\usepackage{tikz}								
\usepackage{autobreak}
\usepackage{color}
\usepackage{amsmath,amscd}	
\usepackage{mathtools}
\usepackage{stmaryrd}
\usepackage{animate}
\usepackage{tikz-cd}
\usepackage{soul}

\usepackage{hyperref}

\usepackage{xypic}

\input xy
\xyoption{all}

\newtheorem {theorem}{Theorem}[section]
\newtheorem {lemma} [theorem] {Lemma}
\newtheorem {proposition} [theorem] {Proposition}

\newtheorem {notation} [theorem] {Notation}

\theoremstyle{definition}
\newtheorem{definition}[theorem]{Definition}
\newtheorem{remark}[theorem]{Remark}
\newtheorem{example}[theorem]{Example}

\begin{document}
\title[]{Bounds on degrees of covers with injective monodromy in iterated Kodaira Fibrations}
\author{Kejia Zhu}
\address{Department of Mathematics, Statistics and Computer Science, University of Illinois at Chicago}
\email{kzhu14@uic.edu}

\date{}
\maketitle

\begin{abstract}
Let $\pi:X\to Y$ be an $n$-dimensional iterated Kodaira fibration with fiber of genus $g$ and injective monodromy. In \cite{Py}, it is proved that we can pass to a finite index subgroup of $\pi_1(X)$ to get the base space of an n+1-dimensional iterated Kodaira fibration with injective monodromy and they asked about bounding the index of such a group. We provide a bound on this index.

\end{abstract}

\section{Introduction:}

We say that an $n$-dimensional compact complex manifold $X$ is an \emph{iterated Kodaira fibration} if there exists a holomorphic submersion $\pi: X\to Y$ with connected fibers which is not isotrivial, where $Y$ is an $(n-1)$-dimensional iterated Kodaira fibration. In particular, when $n=2$, $\pi: X\to Y$ is a Kodaira fibration and $Y$ is a Riemann surface. E.Y. Miller in \cite{Miller86thehomology} showed that one can construct iterated Kodaira fibrations in all dimensions.\\

Given an $n$-dimensional compact complex manifold $X$, consider a non-isotrivial fiber bundle $$F\to X\to Y,$$ where $Y$ is $(n-1)$-dimensional. Since $F$ is closed in $X$, as the preimage of a closed point, $F$ is compact. By the holomorphic version of the regular value theorem, $F$ is a complex manifold of dimension $1$, thus a compact Riemann surface. As in the case of Kodaira fibration, the genus of the fiber $g(F)>2$, see \cite[Theorem 1.1]{kas1968deformations} and \cite[\S14]{barth2015compact} for details.\\

\begin{definition}\cite[Definition 36]{Py}
	A group $G$ is a \emph{polysurface group of length $n$} if there exists a filtration $1=G_0<G_1<...<G_n=G$, so that for each $1\le i\le n$, $G_i$ is a normal subgroup of $G$ and $G_i/G_{i-1}$ is isomorphic to a surface group.
\end{definition}

Historically, F.E.A Johnson in \cite{johnson1994group} proved the following theorem:

\begin{theorem}
	Let $G$ be a group which can be realized in two distinct ways as a polysurface group of length $2$. Assume that at least one of these realizations has infinite monodromy. Then the monodromy homormophism associated to any realization of $G$ as a polysurface group of length $2$ is injective.
\end{theorem}

\noindent The above result implies that double Kodaira fibrations (i.e. fibrations admit two different Kodaira fibrations) have injective monodromy and also enabled the first embedding of surface groups in mapping class groups. Llosa Isenrich and Py proved the following theorem by developing Miller's construction of iterated Kodaira fibration from \cite{Miller86thehomology}:

\begin{theorem}\cite[Theorem 6]{Py}
	For each $n\ge 2$, there exists an $n$-dimensional iterated Kodaira fibration
	$\pi: X\to Y$ with fiber $F$, such that the monodromy representation $\pi_1(Y)\to\text{Mod}(F)$ is injective.
\end{theorem}

\noindent The construction of this theorem provides a higher-dimensional version of Johnson's embedding result. In \cite[Theorem 42]{Py}, starting with an $n$-dimensional Kodaira fibration $X\to Y$ with injective monodromy and fibers of genus $g$, Llosa Isenrich and Py constructed an $n+1$-dimensional Kodaira fibration with injective monodromy by passing $X$ to a finite cover $X_2$. Then they asked the degree of the finite cover $X_2\to X$.\\

The first result of this paper (see Proposition \ref{pro1}) gives an effective version with recursive formula of the following proposition from\cite{Py}:

\begin{proposition}\cite[Proposition 37]{Py}
	Let $G$ be a polysurface group and let $A$ be a finite abelian group. Then any central extension $$1\to A\to \Gamma\to G\to1$$ becomes trivial after passing to a finite index subgroup of $G$.
\end{proposition}


Before we state the result, we need to define some notations:

\begin{notation}
 Given a finite abelian group $A$, let $e(A)$ be the exponent of $A$ and $r(A)$ be a generating set of $A$ of minimal size. If $G$ is a polysurface group of length $n$ with filtration $$1=G_0<G_1<...<G_n=G,$$ let $g_1(G)$ be the genus of the surface group $G_1$.
\end{notation}

The following proposition is an effective version of \cite [Proposition 37]{Py}.\\

\begin{proposition}\label{pro1}
    Let $G$ be a polysurface group of length $n$ and let $A$ be a finite abelian group. There exist:\\
	(1) A subgroup $J$ of the polysurface group $G/G_1$ of index at most
	$$|A|^{2e(A) g_1(G)}\cdot e(A)^2(e(A)!)^{2g_1(G)+|r(A)|}$$ 
	and \\
	(2) A subgroup $K$ of $J$ of index at most $I_{n-1}(A, J)$, where the numbers $I_i(A,\bullet)$ are defined recursively by:\\
	$I_1(A,G)=e(A)$ and if $n>1$ then $$I_n(A,G)=e(A)^3|A|^{2e(A) g_1(G)}(e(A)!)^{2g_1(G)+|r(A)|}\cdot I_{n-1}(A,K)\cdot I_{n-1}(A,J);$$ so that for any central extension $$1\to A\to\Gamma\xrightarrow{\phi} G\to 1,$$ there exists some finite index subgroup $W< G$ so that $[G:W]\le I_n(A,G)$ and the induced extension is trivial.
\end{proposition}

\begin{remark}\label{a}
	 The terms $I_{n-1}(A,J)$ and $I_{n-1}(A,K)$ in the formula for $I_n(A,G)$ can be computed recursively, using Lemma \ref{genus} to bound the genus of the surfaces in their polysurface filtrations. For example, $$I_{1}(A,J)=I_{1}(A,K)=e(A);$$ 
	$$I_{2}(A,J)=e(A)^5|A|^{2e(A) g_1(J)}(e(A)!)^{2g_1(J)+|r(A)|};$$ 
	$$I_{3}(A,J)=e(A)^{13}|A|^{2e(A)\cdot (g_1(J)+g_a(J)+g_b(J))}\cdot(e(A)!)^{6|r(A)|+ (g_1(J)+g_a(J)+g_b(J))},$$ where $$g_a(J)<e(A)^2|A|^{2e(A) g_1(J)}(e(A)!)^{2g_1(J)+|r(A)|}\cdot (g(J_2/J_1)-1)+1,$$ $$g_b(J)<e(A)^7|A|^{2e(A) (g_1(J)+g_a(J))}(e(A)!)^{4|r(A)|+(g_1(J)+g_a(J))} \cdot (g(J_2/J_1)-1)+1.$$ The formula for $I_4(A,J)$ is very complicated, and we omit it.
\end{remark}

The second result (see Theorem \ref{prop}) of this paper answers Llosa Isenrich and Py's question regarding the degree of $X_2\to X$ mentioned earlier with an upper bound. Before we state the result, we need to give the following notations:

\begin{notation}
	For convenience, in the following, we define $$\tau_0(g):=4g+2^{4g}(g-1)+1$$ and define $$\tau(g):=2^{\tau_0}.$$
\end{notation}	
	
Next we are going to introduce three finite covers of $Y$, namely $Z_1,Z_2$ and $Z_3$. The definitions of the three covers involve three homomorphisms $\mu_a,\rho$ and $\mu$ defined in the Section 4. \\

Now we can give the definitions of the covers: $Z_1:=\ker\mu_a$, of degree at most $\prod_{i=1}^{2g-1} (2^{2i}-1)2^{2i-1})$; $Z_2:=\ker\rho$ of degree at most $\prod_{i=1}^{\tau_0(g)} (\tau(g)-2^{i-1})$ and $Z_3:=\ker\mu$ of degree at most $\prod_{i=1}^{g} (2^{2i}-1)2^{2i-1}$. Note since $\pi_1(Z_i)$ ($i=1,2,3$) is a finite index subgroup of $\pi_1(Y)$, it is a polysurface group of length $n-1$. Also note Lemma \ref{genus} allows us to calculate the genera of the surfaces for the $Z_i$ in terms of the degree of the covering and the surfaces for $Y$.

\begin{notation}
	We define $$\kappa_1(n-1):=I_{n-1}((\mathbb Z/2\mathbb Z)^{2g},Z_1),$$

	$$\kappa_2(n-1):=I_{n-1}((\mathbb Z/2\mathbb Z)^{\tau_0(g)},Z_2),$$  
	
	$$\kappa_3(n-1):=I_{n-1}((\mathbb Z/2\mathbb Z)^{4g-2},Z_3),$$
	
	$$\theta(g):=2^{4g-1}\cdot\prod_{i=1}^{2g-1} (2^{2i}-1)2^{2i-1})\cdot\prod_{i=1}^{g} (2^{2i}-1)2^{2i-1}\cdot \prod_{i=1}^{\tau_0(g)} (\tau(g)-2^{i-1}).$$

\end{notation}

\begin{theorem} \label{prop}
	 The finite covering $X_2\to X$ constructed above has degree at most $$ \theta(g)\cdot \kappa_1(n-1)\cdot \kappa_2(n-1)\cdot \kappa_3(n-1).$$
	
\end{theorem}

\noindent\textbf{Acknowledgements.} The author wishes to thank his advisor Daniel Groves and co-adivisor Anatoly Libgober for constant support and warm encouragement. He wishes to thank Claudio Llosa Isenrich and Pierre Py for answering his questions and helpful comments. He is also grateful to the referee for many helpful comments which improved the paper.\\

\section{Preliminaries}

\begin{lemma} 
The fundamental group of an $n$-dimensional iterated Kodaira fibration is a polysurface group of length $n$.
\end{lemma}
\begin{proof} 
Recall a surface group is polysurface and by the defnition of $n$-dimensional iterated Kodaira fibration, there is a fiber bundle $$F\to K_n\to K_{n-1},$$ where $F$ is a hyperbolic surface and $K_n$ is the $n$-dimensional iterated Kodaira fibration. So there is a short exact sequence:
$$1\to \pi_1(F)\to \pi_1 (K_n)\to\pi_1(K_{n-1})\to 1.$$

For the $2$-dimensional iterated Kodaira fibration case, there is a fiber bundle $$F\to K_2\to E,$$ where $E,F$ are hyperbolic surfaces. Thus $K_2$ is a polysurface group of length $2$. \\

Now we claim that given a short exact sequence $$1\to F\to G\xrightarrow{\pi} G_0\to 1,$$ 
where $F$ is a surface group and $G_0$ a polysurface group of length $k$, then $G$ is a polysurface group of length $k+1$. Indeed, by assumption, there is a filtration $$1\unlhd H_1\unlhd H_2...\unlhd H_k=G_0$$ so that $H_{i+1}/H_i$ is a surface group. Now we define $F_1:=F$ and $F_i:=\pi^{-1}(H_{i-1})$, it is easy to see that $F_i\unlhd F_{i+1}$. Observe $F_{i+1}/F_i\cong H_{i+1}/H_i$, which is a surface group, thus $G$ is a polysurface group of length $k+1$.\\

By induction, the fundamental group of an $n$-dimensional iterated Kodaira fibration is a polysurface group of length $n$. 
\end{proof}

\begin{notation}
If a group $V$ is the fundamental group of a closed surface, say $\Sigma$, then we denote the genus of $\Sigma$ by $g(V)$ and denote the Euler characteristic of $\Sigma$ by $\chi(V)$.
\end{notation}

\begin{lemma}\label{genus}
	Let $G$ be polysurface group of length $n$ with filtration $1=G_0<G_1<...<G_n=G$ and $D$ be a subgroup of $G$ of finite index $d$. Then $D$ is a polysurface group of length $n$ with filtration $1=D_0<D_1<...<D_n=D$, where $D_i:=D\cap G_i$. Moreover, we have
	 $$g(D_i/D_{i-1})\le d\cdot (g(G_i/G_{i-1})-1)+1.$$ 
\end{lemma}
\begin{proof}
	By the Second Isomorphism Theorem, 
	$$\frac{D_i}{D_{i-1}}=\frac{D_i}{D_i\cap G_{i-1}}\cong \frac{D_i\cdot G_{i-1}}{G_{i-1}}.$$
	Note $$[G_i:D_i\cdot G_{i-1}]\le [G_i:D_i]\le [G:D]=d,$$ thus $\frac{D_i}{D_{i-1}}$ is a subgroup of $\frac{G_i}{G_{i-1}}$ of index at most $d$. In particular, it is a surface group. Then $$\chi(\frac{D_i}{D_{i-1}})=[\frac{G_i}{G_{i-1}}:\frac{D_i}{D_{i-1}}]\cdot \chi(\frac{G_i}{G_{i-1}})\ge d\cdot \chi(\frac{G_i}{G_{i-1}}),$$ which implies $g(D_i/D_{i-1})\le d\cdot (g(G_i/G_{i-1})-1)+1$. 
\end{proof}

\section{An elaboration of \cite[Proposition 37]{Py}}
In this section, we prove Proposition \ref{pro1}. We start in the case of length $1$ (see Lemma \ref{le1}). Note that the length $2$ is the case of classical a Kodaira fibration.

\begin{lemma} \label{le1}
If $\mathcal G_1$ is a surface group and $1\to A\to \Gamma\to \mathcal G_1\to 1$ is a central extension, then there is a finite index subgroup $H<\mathcal G_1$ with $[\mathcal G_1:H]\le I_1(A,\mathcal G_1):= e(A)$ so that the induced extension of $H$ by $A$ is trivial. Moreover, we can take $H\unlhd \mathcal G_1$ and $\mathcal G_1/H$ to be cyclic.
\end{lemma}
\begin{proof}
The extension corresponds to a class in $H^2(\mathcal G_1, A)$. Since $H_1(\mathcal G_1,\mathbb Z)$ is torsion-free, the universal coefficient theorem implies
$$H^2(\mathcal G_1,A)=\text{Hom}(H_2(\mathcal G_1, \mathbb Z),A).$$

Let $\mathcal G_1:=\pi_1(\Sigma)$ for some surface $\Sigma$. If $H< \mathcal G_1$ is a subgroup of index $d$, then $H=\pi_1(\Sigma_1)$, where $f:\Sigma_1\to\Sigma$ is a degree $d$ covering map inducing the inclusion $\pi_1(f):H\to \mathcal G_1$. Then the induced degree map 
$$f_*:H_2(H,\mathbb Z)\to H_2(\mathcal G_1,\mathbb Z)$$ is 
$$\mathbb Z\xrightarrow{\times d}\mathbb Z.$$
Thus $$f^*H^2(H,A)=\text{Hom}(f_*H_2(H, \mathbb Z),A)=\{\phi:x\mapsto \phi(dx)\in A)\},$$ where $\phi\in\text{Hom}(H_2(\mathcal G_1, \mathbb Z),A)$. Since $A$ is abelian, $\phi(dx)=d\phi(x)$.\\

Note we can make $H$ normal by considering a surjection from the surface group to a cyclic group of order $e(A)$. Thus we may require $\mathcal G_1/H=\mathbb Z/e(A)\mathbb Z$ and construct the corresponding $\Sigma_1$. So $$f^*H^2(H,A)=\{\phi:x\mapsto \phi(dx)=d\phi(x)=1\in A\},$$ where $\phi\in\text{Hom}(H_2(\mathcal G_1, \mathbb Z),A)$. Hence $f^*H^2(H,A)$ is trivial. Thus the corresponding extension is trivial, in other words, when restricting $\mathcal G_1$ to a subgroup of index $e(A)$ in the extension $$1\to A\to\Gamma\to \mathcal G_1\to 1,$$ there is a trivial extension. 
\end{proof}

\begin{remark}
	We have proved Proposition \ref{pro1} in case $n=1$.
\end{remark}


Let $G$ be a polysurface group of length $n$ with central extension $1\to A\to\Gamma\xrightarrow{\phi} G\to1$ and let $G_1< G$ be a normal surface subgroup of $G$. By Lemma \ref{le1}, there is a finite index normal subgroup of  $G_1$, say $G_1'$, with index at most $e(A)$, such that the extension is trivial over $G_1'$ and the quotient group $G_1/G_1'$ is cyclic. So there is a pair of short exact sequences

$$ \begin{tikzcd}
	1\arrow[r]  & A \arrow[r] &\Gamma \arrow[r,"\phi"] &G \arrow[r]& 1\\
	1\arrow[r]  & A \arrow[r]\arrow[u] &\Gamma_1'\cong A\oplus G_1' \arrow[r]\arrow[u,hook] & G_1' \arrow[r]\arrow[u,hook]& 1
\end{tikzcd}$$\\

Let $i:G_1'\to\Gamma$ be a lift of $G_1'$ and let $N_\Gamma(i(G_1'))$ be the normalizer of $i(G_1')$ in $\Gamma$.\\

\begin{lemma} \label{le2}
$N_\Gamma(i(G_1'))$ is finite index in $\Gamma$ with index less than $|A|^{2e(A) g_1(G)}\cdot e(A)(e(A)!)^{2g_1(G)+|r(A)|}$.
\end{lemma}
\begin{proof}
With the central extension $1\to A\to\Gamma\xrightarrow{\phi} G\to 1$, define $$\Gamma_1:=\phi^{-1}(G_1)\unlhd\Gamma$$ and $$\Gamma_1':=\phi^{-1}(G_1')\cong A\oplus i(G_1').$$ Since $[G_1:G_1']\le e(A)$, it follows that $[\Gamma_1':i(G_1')]=|A|$ and $[\Gamma_1:\Gamma_1']\le e(A)$. Define $N:=N_\Gamma(\Gamma_1')$, observe $N_\Gamma(i(G_1'))=N_N(i(G_1'))$, thus $$[\Gamma:N_\Gamma(i(G_1'))]=[\Gamma:N]\cdot [N:N_N(i(G_1'))].$$
We first consider $[\Gamma:N]$. Define an action of $\Gamma$ on the subgroups of index $[\Gamma_1:\Gamma_1']$ in $\Gamma_1$ by conjugation. Note $N=N_\Gamma(\Gamma_1')=\text{Stabilizer}_\Gamma(i(\Gamma_1'))$, hence $$[\Gamma:N] =|\text{Orbit}_\Gamma(\Gamma_1')|=|\text{Orbit}_{\Gamma_1}(\Gamma_1')|,$$
which is bounded by the number of subgroups of index $[\Gamma_1:\Gamma_1']$ in $\Gamma_1$.\\

Observe $\Gamma_1$, as a finite extension of a surface group, is finitely generated. So $\Gamma_1$ has only finitely many subgroups of a given index. Indeed, by \cite{HM}, if $N(d,n)$ is the number of subgroups of index $d$ in the free group of rank $n$, then
$$N(d,n)=d(d!)^{n-1}-\sum_{i=1}^{d-1}((d-i)!)^{n-1}N(i,n)<d(d!)^{n-1}.$$

Note there is a natural epimorphism $p:F_{|Y_n|}\to \Gamma_1$, where $Y_n$ is a generating set of $\Gamma_1$ and $F_{|Y_n|}$ is the free group with rank $|Y_n|$. Note that if $H$ is a finite index subgroup of $\Gamma_1$, then $$[F_{|Y_n|}:p^{-1}(H)]=[\Gamma_1:H].$$ So we can pull back the subgroups of index $[\Gamma_1:\Gamma_1']$ in $\Gamma_1$ to the subgroups of the same index in $F_{|Y_n|}$. So there are at most $N([\Gamma_1:\Gamma_1'],|Y_n|)$ subgroups of index $[\Gamma_1:\Gamma_1']$ in $\Gamma_1$. Since $\Gamma_1 /A = G_1$ and $G_1$ has a generating set of size $2g_1(G)$, it follows that $$|Y_n|< 2g_1(G)+|r(A)|.$$ Therefore,
\begin{align}
	[\Gamma:N] &=|\text{Orbit}(\Gamma_1')|\nonumber\\
	&\le N([\Gamma_1:\Gamma_1'],|Y_n|)\nonumber\\
	&<[\Gamma_1:\Gamma_1'])([\Gamma_1:\Gamma_1']!)^{2g_1(G)+|r(A)|-1}.\nonumber\\
	&<e(A)(e(A)!)^{2g_1(G)+|r(A)|}.\nonumber
\end{align}

Now we start to deal with $[N:N_N(i(G_1'))]$. Applying a similar argument as above to $[\Gamma:N]=[\Gamma:N_\Gamma(\Gamma_1')]$, define an action of $N$ on the subgroups of index $[\Gamma_1':i(G_1')]=|A|$ in $\Gamma_1'$ by conjugation. It follows that $$[N:N_N(i(G_1'))]=|\text{Orbit}_{\Gamma_1'}(i(G_1'))|,$$ which is bounded by the number of subgroups of $\Gamma_1'$ isomorphic to $i(G_1')$. Recall $i(G_1')$ is a surface group of genus, say $g'$. Fix an isomorphism $\Gamma_1'\cong A\oplus i(G_1')$ and
choose a generating set of $i(G_1')$ to be $$\{(0,a_1),...,(0,a_{g'}),(0,b_1),...,(0,b_{g'})\}.$$

We claim that any subgroup of $A\oplus i(G_1')$ isomorphic to $i(G_1')$ is of the form $$\langle(x_1,a_1),...(x_{g'},a_{g'}),(y_1,b_1),...,(y_{g'},b_{g'})\rangle,$$ where  $x_i,y_i$ are arbitrary in $A$.\\

Consider the projection $p:A\oplus i(G_1')\to i(G_1')$. For any subgroup $K$ of $A\oplus i(G_1')$ isomorphic to $i(G_1')$, since $\ker p|_K \subset \ker p=A$ is a finite abelian group as well as a subgroup of a surface group, it follows that $\ker p|_K$ is trivial. Recall $p(K)\cong K, i(G_1')$ are isomorphic surface groups, so $p(K)=i(G_1')$. Therefore, $K$ contains a subset of the type in the listed generating set from the statement of the claim. But any such subset generates a subgroup isomorphic to $i(G_1')$, so the result of the claim follows.\\

Now with the claim, we know the subgroups of $i(\Gamma_1')$ of index $|A|$, isomorphic to $i(G_1')$, is bounded by $|A|^{2g'}$. Recall $[G_1:G_1']\le e(A)$, thus $g'\le e(A)\cdot g_1(G)$. So $$[N:N_N(i(G_1'))]\le |A|^{2e(A) g_1(G)}.$$ Therefore, we know $$[\Gamma:N_N(i(G_1'))]<|A|^{2e(A) g_1(G)}\cdot e(A)(e(A)!)^{2g_1(G)+|r(A)|}.$$

\end{proof}


\begin{lemma}\label{3use}
	Let $G$ be a polysurface group of length $n$ with filtration $1=G_0\unlhd G_1\unlhd G_2\unlhd...\unlhd\mathcal G_n=G$. For any central extension 
	$$1\to A\to \Gamma\xrightarrow{\phi} G\to 1,$$ there is a finite index subgroup $G_1'$ of $G_1\subset G$, so that $G_1/G_1'$ is cyclic with order $e(A)$ and the induced extension $$1\to A\to \Gamma_1'\to  G_1'\to 1$$ is trivial. Moreover, the index of the subgroup $\Gamma':=N_\Gamma(i(G_1'))$ of $\Gamma$ is less than
	 $|A|^{2e(A) g_1(G)}\cdot e(A)(e(A)!)^{2g_1(G)+|r(A)|}$.
\end{lemma}

\begin{proof}
	By assumption $G_1$ is a surface group, thus by Lemma \ref{le1}, there is a finite index subgroup $G_1'$ of $G_1$ so that $G_1/G_1'$ is cyclic with order at most $e(A)$ and the extension restricted to $G_1'$ is trivial. The inclusion $$i: G_1'\hookrightarrow{}A\oplus G_1'\subset\Gamma$$ defines an inclusion $G_1'\hookrightarrow{}\Gamma$. By Lemma \ref{le2}, $\Gamma'$ is of finite index in $\Gamma$, with index less than $|A|^{2e(A) g_1(G)}\cdot e(A)(e(A)!)^{2g_1(G)+|r(A)|}$.
\end{proof}

\begin{proof}[Proof of Proposition \ref{pro1}:]
The proof is by induction. Lemma \ref{le1} is the length $1$ case. Now suppose the length is $n>1$. By Lemma \ref{3use}, there is a finite index subgroup $G_1'$ of $G_1$ so that $G_1/G_1'$ is cyclic, the induced extension $$1\to A\to \Gamma_1'\to  G_1'\to 1$$ is trivial and there is a subgroup $\Gamma':=N_\Gamma(i(G_1'))$ of $\Gamma$ so that $$[\Gamma:\Gamma']<|A|^{2e(A) g_1(G)}\cdot e(A)(e(A)!)^{2g_1(G)+|r(A)|}.$$ Now define $G':=\phi(\Gamma')$ and consider the induced extension:
$$1\to A\to \Gamma'\to G'\to 1.$$

By taking the quotient by $i(G_1')$, we obtain a central extension $$1\to A\to \Gamma'/i(G_1')\to G'/G_1'\to 1.$$

We claim that $G'/G_1'$ is an extension of a polysurface group of length $n-1$ by a finite abelian group. First, there is a short exact sequence
$$1\to (G_1\cap G')/G_1'\to G'/G_1' \xrightarrow{h} G'/(G_1\cap G')\to 1.$$
Note $$(G_1\cap G')/G_1'<G_1/G_1',$$ so $(G_1\cap G')/G_1'$ is a finite abelian group. Observe $G'/(G_1\cap G')$ is a finite index subgroup of the polysurface group of length $n-1$, $G/G_1$: since $G_1\cap G'$ is the kernel of $G'\to  G/G_1$ and $G'/(G_1\cap G')$ is the image of $G'$ under $G\to  G/G_1$, it follows that $$G'/(G_1\cap G')<G/G_1.$$ Moreover, since $G'$ is of finite index in $G$, its image $G'/(G_1\cap G')$ is of finite index in $G/G_1$. More precisely,

  \begin{align}
  	
  	[G/G_1:G'/(G_1\cap G')]&\le [G/(G_1\cap G'):G'/(G_1\cap G')]\nonumber\\
  	&=[G:G']\nonumber\\
  	&=[\Gamma:\Gamma']\nonumber\\
  	&<|A|^{2e(A) g_1(G)}\cdot e(A)(e(A)!)^{2g_1(G)+|r(A)|}\nonumber
  \end{align}

Let $H$ be the kernel of the natural map $$G'/G_1'\to\text{Aut}((G_1\cap G')/G_1').$$ Since $(G_1\cap G')/G_1'$ is finite, so is $\text{Aut}((G_1\cap G')/G_1')$, hence $[G'/G_1':H]$ is finite.  Also recall $(G_1\cap G')/G_1'$ is abelian, so $(G_1\cap G')/G_1'\subset H$ since it acts trivially on itself via conjugation. So $H$ is a central extension of a polysurface group of length $n-1$ by a finite abelian group:
$$1\to (G_1\cap G')/G_1'\to H \to h(H)\to 1.$$

In particular, 
\begin{align}
[G'/(G_1\cap G'):h(H)]&=[G'/G_1':H]\nonumber\\
&=|\text{Aut}((G_1\cap G')/G_1')|<e(A).\nonumber
\end{align}

Therefore, $h(H)$ is a subgroup of polysurface group $G/G_1$ of length $n-1$ with index strictly smaller than $$ |A|^{2e(A) g_1(G)}\cdot e(A)^2(e(A)!)^{2g_1(G)+|r(A)|}.$$

Applying the induction hypothesis, there is a subgroup $H_0$ of $H$ of index at most $$I_{n-1}((G_1\cap G')/G_1', h(H))$$ with a trivial central extension:
$$1\to (G_1\cap G')/G_1'\to H_0\to H_0'\to 1.$$ 

Note the trivial extension gives an embedding of $H_0'$ into $H_0$ as a finite index subgroup with $$[H_0:H_0']=|(G_1\cap G')/G_1'|.$$ Recall $H_0$ is a finite index subgroup of $H$, $H$ is a finite index subgroup of polysurface group of length $n-1$: $G/G_1$, thus $H_0'$ is a finite index subgroup of the polysurface group $G'/G_1'$ of length $n-1$.\\

In particular, $H_0'<H_0<H< G'/G_1'$, with 
\begin{align}
	[G'/G_1':H_0']&=[G_1/G_1':H]\cdot [H:H_0]\cdot [H_0:H_0'] \nonumber\\
	& <[G_1/G_1':H]\cdot  I_{n-1}((G_1\cap G')/G_1', h(H))\cdot|(G_1\cap G')/G_1'| \nonumber\\
	& \le |\text{Aut}((G_1\cap G')/G_1')|\cdot I_{n-1}((G_1\cap G')/G_1', h(H))\cdot e(A) \nonumber\\
	& <I_{n-1}((G_1\cap G')/G_1', h(H))\cdot e(A)^2 \nonumber
\end{align}
(recall $|G_1/G_1'|=e(A)$). \\


Moreover, since $H_0'$ is a finite index subgroup of the length $n-1$ polysurface group $h(H)$ of index 
\begin{align}
	[h(H):H_0']&=[H:H_0] \nonumber\\
	& \le I_{n-1}((G_1\cap G')/G_1', h(H)).\nonumber
\end{align}
It follows that $H_0'$ is a polysurface group of length $n-1$.\\

Now, coming back to $$1\to A\to \Gamma'/i(G_1')\to G'/G_1'\to 1,$$ restrict $G'/G_1'$ to $H_0'$. Then we get a central extension of a polysurface group of length $n-1$ by a finite abelian group:

$$1\to A\to \Delta_n\to H_0'\to 1,$$ 
where $\Delta_n\subset \Gamma'/i(G_1')$ is the preimage of $H_0'$ in $\Gamma'/i(G_1')$. Thus 
\begin{align}
[\Gamma'/i(G_1'):\Delta_n]&=[G'/G_1':H_0']\nonumber\\
                          &\le I_{n-1}((G_1\cap G')/G_1', h(H))\cdot e(A)^2. \nonumber
\end{align}
Using the induction hypothesis one more time on this new extension, there is a finite index subgroup $M<\Delta_n$ of index at most $I_{n-1}(A,H_0')$, containing $A$, on which the extension is trivial. So there is a finite index subgroup $M<\Gamma'/i(G_1')$ containing $A$, of index at most $$I_{n-1}(A,H_0')\cdot I_{n-1}((G_1\cap G')/G_1', h(H))\cdot e(A)^2,$$ and on which the extension is trivial. This implies that there is a left splitting $M\to A$ for the inclusion $A\to M$. It induces a left-splitting $\Gamma''\to A$ of the preimage $\Gamma''$ of $M$ in $\Gamma'$ with
\begin{align}
 [\Gamma':\Gamma'']&=[\Gamma'/i(G_1'):M]\nonumber\\
                   & <I_{n-1}(A,H_0')\cdot I_{n-1}((G_1\cap G')/G_1', h(H))\cdot e(A)^2. \nonumber
\end{align}
Thus, $\Gamma''<\Gamma$ is a finite index subgroup on which the induced extension
$$1\to A\to \Gamma''\to W\to 1$$
 is trivial, where $W:=\phi(\Gamma'')$. So
\begin{align}

	   [G:W] &=[\Gamma:\Gamma'']\nonumber\\
	         &=[\Gamma:\Gamma'][\Gamma':\Gamma'']\nonumber\\
	         &<|A|^{2e(A) g_1(G)}\cdot e(A)(e(A)!)^{2g_1(G)+|r(A)|}\cdot I_{n-1}(A,H_0')\cdot I_{n-1}((G_1\cap G')/G_1',h(H))\cdot e(A)^2 \nonumber\\
	         &=|A|^{2e(A) g_1(G)}\cdot e(A)^3(e(A)!)^{2g_1(G)+|r(A)|}\cdot I_{n-1}(A,H_0')\cdot I_{n-1}((G_1\cap G')/G_1',h(H)).\nonumber
\end{align}

	


Note $(G_1\cap G')/G_1'$ is a subgroup of $G_1/G_1'$ of order smaller than $e(A)$. Recall $h(H)$ is a subgroup of polysurface group $G/G_1$ of length $n-1$ with index strictly smaller than $$|A|^{2e(A) g_1(G)}\cdot e(A)^2(e(A)!)^{2g_1(G)+|r(A)|},$$ we replace $h(H)$ by a subgroup of polysurface group $G/G_1$, say $J$, of index at most $$|A|^{2e(A) g_1(G)}\cdot e(A)^2(e(A)!)^{2g_1(G)+|r(A)|}.$$ Recall $$[h(H):H_0']\le I_{n-1}((G_1\cap G')/G_1', h(H)),$$ so we replace $H_0'$ by a subgroup of $J$, say $K$, of index $$I_{n-1}((G_1\cap G')/G_1', J)\le I_{n-1}(A, J).$$ Therefore, there is a recursive formula:
$$I_1(A,G)= e(A) $$ and $$I_n(A,G)=e(A)^3|A|^{2e(A) g_1(G)}(e(A)!)^{2g_1(G)+|r(A)|}\cdot I_{n-1}(A,K)\cdot I_{n-1}(A,J),$$
where the first inequality follows from Lemma \ref{le1}. \end{proof}

\section{The proof Theorem \ref{prop} and examples}
Let $X$ be an $n$-dimensional iterated Kodaira fibration with injective monodromy and fibers of genus $g$. In \cite{Py}, Llosa Isenrich and Py constructed a finite covering space $X_2$ of $X$ and a family $Z^*\to X_2$ of closed Riemann surfaces of genus $2 + 4(g- 1)$ whose monodromy representation is injective. They asked what the minimal degree of the covering $X_2 \to X$ of the family $Z^* \to X_2$ is? We give an upper bound.\\


First we review the construction of the $X_2$ above. Given an iterated Kodaira fibration of dimension $n$ with injective monodromy $\pi: X\to Y$ (here by definition, $Y$ is a Kodaira fibration of dimension $n-1$), denote the fundamental group of the fiber (a Riemann surface of genus $g$) by $R$ and denote the kernel of the natural map $R\to H_1(R,\mathbb Z/2\mathbb Z)$ by $R_2$. Then we have (16) of \cite{Py}:
$$1\to H_1(R,\mathbb Z/2\mathbb Z)\to\pi_1(X)/R_2\to\pi_1(Y)\to 1.$$
By Proposition \ref{pro1}, there is a finite cover $Y_a$ of $Y$ so that if $\pi':X_a\to Y_a$\\

\begin{center}
	\begin{tikzcd}
		X_a\arrow[r] \arrow[d,"\pi'"]
		& X \arrow[d,"\pi"] \\
		Y_a\arrow[r]
		& Y
	\end{tikzcd}
\end{center}

\noindent is the pullback then $1\to H_1(R,\mathbb Z/2\mathbb Z)\to\pi_1(X_a)/R_2\to\pi_1(Y_a)\to 1$
is trivial.\\

We then have $$\pi_1(X_a)/R_2\cong \pi_1(Y_a)\oplus H_1(R,\mathbb Z/2\mathbb Z).$$ It gives a surjection $\phi:\pi_1(X_a)\to Z/2\mathbb Z$ which is non-trivial on $R$. Let $X_b$ be the covering space of $X_a$ corresponding to the kernel of $\phi$. The degree of the covering map $f:X_b\to X_a$ is $2$.\\

\begin{lemma}\label{3.1}
	The degree of $X_a\to X$ is at most $$\prod_{i=1}^{g} (2^{2i}-1)2^{2i-1}\cdot \kappa_3(n-1),$$
	where $\mu$ is given by $\mu:\pi_1(Y)\to \mathrm{Aut}(H_1(R,\mathbb Z/2\mathbb Z))$.
\end{lemma}
\begin{proof}
	Recall there is an extension $$1\to H_1(R,\mathbb Z/2\mathbb Z)\to\pi_1(X)/R_2\xrightarrow{\psi} \pi_1(Y)\to 1.$$ Pass to the kernel of $$\mu:\pi_1(Y)\to \text{Aut}(H_1(R,\mathbb Z/2\mathbb Z))=\text{GL}(2g,\mathbb Z/2\mathbb Z).$$ Since the image is contained in $\text{Sp}(2g,\mathbb Z/2\mathbb Z)$, the kernel has index at most $$|\text{Sp}(2g,\mathbb Z/2\mathbb Z)|=\prod_{i=1}^{g} (2^{2i}-1)2^{2i-1}.$$ Observe that there is an induced central extension 
	$$1\to H_1(R,\mathbb Z/2\mathbb Z)\to\psi^{-1}(\ker\mu)\xrightarrow{\psi|_{\psi^{-1}(\ker\mu)}}\ker\mu\to 1.$$
	
	Apply the Proposition \ref{pro1} to this new induced central extension. By observing $\ker\mu$ is a finite index subgroup of $\pi_1(Y)$, a polysurface group of length $n-1$, we get the finite index subgroup $\pi_1(X_a)$ of $\pi_1(X)$ with index at most $$I_{n-1}(H_1(R,\mathbb Z/2\mathbb Z),\ker\mu).$$ Therefore, the degree of $X_a\to X$ is at most $$\prod_{i=1}^{g} (2^{2i}-1)2^{2i-1}\cdot I_{n-1}(H_1(R,\mathbb Z/2\mathbb Z),\ker\mu).$$  Note $I_{n-1}(H_1(R,\mathbb Z/2\mathbb Z),\ker\mu)=\kappa_3(n-1)$, so the result follows.
\end{proof}

It is obvious that the degree of $\pi_b:X_b\to X$ is at most $$2\cdot\prod_{i=1}^{g} (2^{2i}-1)2^{2i-1}\cdot \kappa_3(n-1).$$ Moreover, $\pi_b:X_b\to Y_a$ has the structure of
an $n$-dimensional iterated Kodaira fibration with injective monodromy and $X_b$ carries a fixed point free involution $\sigma: X_b\to X_b$ such that $\pi_b \circ \sigma =\pi_b$ and the fibers of $\pi_b$ have genus $1+2(g-1)$. See \cite[Proposition 40]{Py} for the details.\\

Now denote the fundamental group of the fiber, which is a Riemann surface of genus $1+2(g-1)$, of $$\pi_b:X_b\to Y_a$$ by $R'$ and denote the kernel of the natural map $$R'\to H_1(R',\mathbb Z/2\mathbb Z)$$ by $R_2'$, then we have (16) of \cite{Py}:
$$1\to H_1(R',\mathbb Z/2\mathbb Z)\to\pi_1(X)/R_2'\to\pi_1(Y)\to 1.$$

By Proposition \ref{pro1}, there is a finite cover $Y'$ of $Y_a$ so that if $\pi':X'\to Y'$\\

\begin{center}
	\begin{tikzcd}
		X'\arrow[r] \arrow[d,"\pi'"]
		& X_b \arrow[d,"\pi_b"] \\
		Y'\arrow[r]
		& Y_a
	\end{tikzcd}
\end{center}
 
\noindent is the pullback then $$1\to H_1(R',\mathbb Z/2\mathbb Z)\to\pi_1(X')/R_2'\to\pi_1(Y')\to 1$$
is trivial.\\

We then have $$\pi_1(X')/R_2'\cong \pi_1(Y')\oplus H_1(R',\mathbb Z/2\mathbb Z),$$ which gives a surjection $$\phi:\pi_1(X')\to \pi_1(X')/R_2'\to H_1(R',\mathbb Z/2\mathbb Z).$$ Let $X''$ be the covering space of $X'$ corresponding to the kernel of $\phi$. The degree of the covering map $f:X''\to X'$ is $|H_1(R',\mathbb Z/2\mathbb Z)|=2^{4g-2}$ (recall $R'$ is a surface group of genus $2g-1$).\\

Recall there is a holomorphic fixed point free involution $\sigma_b:X_b\to X_b$ so that $\pi_b\circ\sigma=\pi_b$. Let $\sigma': X'\to X'$ be the lift of $\sigma$ which preserves the map $\pi'$. Let $D$ be the
union of the graph of $f$ and of the graph of $\sigma'\circ f$. Note that $D$ naturally sits as a smooth divisor inside the fiber product
$Z:=X''\times_{Y'}X'$. So we can construct a new bundle $Z\setminus D\to Y'$. We denote by $L$ the fundamental group of the fiber of this new
bundle, and by $L_2$ the kernel of the natural map $L\to H_1(L, \mathbb Z/2\mathbb Z)$. Then as (16) of \cite{Py}, we 
can construct an extension:
$$1\to H_1(L,\mathbb Z/2\mathbb Z)\to \pi_1(Z\setminus D)/L_2\to\pi_1(Y')\to 1$$  (\cite[Proposition 41]{Py} ).

\begin{lemma} 
$|H_1(L,\mathbb Z/2\mathbb Z)|=2^{4g+2^{4g}(g-1)+1}$.
\end{lemma}
\begin{proof}
Observe that $L$ is $Z_y\setminus (Z_y\cap D)=X''_y\times X'_y\setminus (\Gamma\cup\Gamma^*)$, where $y\in Y'$ and $$\Gamma:=\{(f(u),u|u\in X_y'')\}, \Gamma^*:=\{(\sigma'\circ f(u),u|u\in X_y')\}$$ for the covering map $f:X''\to X'$ and $\sigma'$ is an involution. By \cite[Lemma on p.209]{Kodaira1967ACT}, $$H_1(Z_y\setminus (Z_y\cap D)\cong H_1(X''_y)\oplus H_1(X'_y)\oplus\mathbb Z/2\mathbb Z.$$ Observe the fiber of $X'\to Y'$ is the same as the fiber of $X_b\to Y_a$, which is a compact Riemann surface of genus $2g-1$. Thus $$H_1(X'_y,\mathbb Z/2\mathbb Z)=(\mathbb Z/2\mathbb Z)^{4g-2}.$$ Now $X''\to X'$ is a $2^{4g-2}$-degree cover, which implies $X''_y\to X'_y$ is a $2^{4g-2}$-cover. So $X''_y$ is of genus $$2^{4g-2}(2g-2)+1,$$ and $$H_1(X''_y,\mathbb Z/2\mathbb Z)={(\mathbb Z/2\mathbb Z)}^{2^{4g}(g-1)+2}.$$ So $$|H_1(L,\mathbb Z/2\mathbb Z)|=2^{4g+2^{4g}(g-1)+1}.$$
\end{proof}

Now apply Proposition \ref{pro1} again to pass to a finite cover $Y''$ of $Y'$ so that the restriction of the last extension to $\pi_1(Y')$ is trivial. Now, define $X_1,X_2$ via pullbacks:\\

\begin{tikzcd}
	X_1\arrow[r,"q_1"] \arrow[d]
	& Y'' \arrow[d] \\
	X'\arrow[r]
	& Y'
\end{tikzcd} 
\begin{tikzcd}
	X_2\arrow[r,"q_2"] \arrow[d]
	& Y'' \arrow[d] \\
	X''\arrow[r]
	& Y'
\end{tikzcd} \\

The composition of the sequence of finite covers $$X_2\to X''\to X'\to X_b\to X$$ is our finite cover $X_2\to X_b$. Now we compute the degree of $X_2\to X_b$ as follows:\\

\noindent We have seen the degree of $X_b\to X$ is bounded by $$2\cdot\prod_{i=1}^{g} (2^{2i}-1)2^{2i-1}\cdot \kappa_3(n-1).$$\\

\noindent We already know $X''\to X'$ is of degree $2^{4g-2}$. \\

\noindent For the degree of $X'\to X_b$, with the same analysis as Lemma \ref{3.1}, the degree of $X'\to X_b$ is at most $$\prod_{i=1}^{2g-1} (2^{2i}-1)2^{2i-1})\cdot I_{n-1}(H_1(R',\mathbb Z/2\mathbb Z),\ker\mu_a)=\prod_{i=1}^{2g-1} (2^{2i}-1)2^{2i-1})\cdot \kappa_1(n-1),$$
where $\mu_a$ is given by $\mu_a:\pi_1(Y_a)\to \mathrm{Aut}(H_1(R',\mathbb Z/2\mathbb Z))=\text{GL}(4g-2,\mathbb Z/2\mathbb Z)$.\\

\noindent Start with the extension $$1\to H_1(L,\mathbb Z/2\mathbb Z)\to \pi_1(Z-D)/L_2\xrightarrow{\varphi}\pi_1(Y')\to 1.$$
Again, pass to a subgroup of $\pi_1(Y')$,

For the degree of $X_2\to X''$, we have the following lemma.\\

\begin{lemma} 
The degree of $X_2\to X''$ is at most $$\prod_{i=1}^{\tau_0(g)} (\tau(g)-2^{i-1})\cdot \kappa_2(n-1),$$ where $\rho$ is given by $\rho:\pi_1(Y')\to \text{Aut}(H_1(L,\mathbb Z/2\mathbb Z)).$
\end{lemma}
\begin{proof}
Since it is a pullback, the degree of $X_2\to X''$ is the same as the degree of $Y''\to Y'$. Start with the extension $$1\to H_1(L,\mathbb Z/2\mathbb Z)\to \pi_1(Z-D)/L_2\xrightarrow{\varphi}\pi_1(Y')\to 1.$$
Again, pass to a subgroup of $\pi_1(Y')$, the kernel of $\rho$ with index at most 

\begin{align}
|\text{Aut}(H_1(L,\mathbb Z/2\mathbb Z))|&=|\text{Aut}((\mathbb Z/2\mathbb Z)^{4g+2^{4g}(g-1)+1})|\nonumber\\
&=\prod_{i=1}^{\tau_0(g)} (\tau(g)-2^{i-1}), \nonumber
\end{align}

 so that the induced extension 
$$1\to H_1(L,\mathbb Z/2\mathbb Z)\to\varphi^{-1}(\ker\rho)\xrightarrow{\varphi|_{\varphi^{-1}(\ker\rho)}}\ker\rho\to 1$$

is central. Now apply Proposition \ref{pro1} to $\pi_1(Y')$ to get $\pi_1(Y'')$, so that the index is at most $$I_{n-1}( H_1(L,\mathbb Z/2\mathbb Z),\ker\rho)=\kappa_2(n-1).$$ Therefore, the degree of $X_2\to X''$ is at most $$\prod_{i=1}^{\tau_0(g)} (\tau(g)-2^{i-1})\cdot \kappa_2(n-1).$$ 
\end{proof}

Now we can conclude that the minimal degree of $$X_2\to X''\to X'\to X_b\to X,$$ is bounded by
$$ \theta(g)\cdot \kappa_1(n-1)\cdot \kappa_2(n-1)\cdot \kappa_3(n-1),$$ which is the Theorem \ref{prop}.

\begin{remark}
By the last section, there is a recursive formula for $I_n(A,G)$:  \\
$I_1(A,G)= e(A)$ and 
\noindent\begin{align}
	I_n(A,G)&=e(A)^3|A|^{2e(A) g_1(G)}(e(A)!)^{2g_1(G)+|r(A)|}\cdot I_{n-1}(A,H_0')\cdot I_{n-1}(A,h(H))\nonumber\\
            &=8\cdot|A|^{2e(A) g_1(G)}\cdot {2}^{2g_1(G)+|r(A)|}\cdot I_{n-1}(A,H_0')\cdot I_{n-1}(A,h(H)),\nonumber
\end{align}	
where the second equality follows from the observation that in this section we are only considering the groups with $e(A)=2$.

\end{remark}

\begin{notation}
	For the convenience in the expression of the following examples, we define
	$$g_3(\ker\mu):=g_1(\ker\mu)+g_a(\ker\mu)+g_b(\ker\mu),$$ (resp. $g_3(\ker\rho)$ and $g_3(\ker\mu_a)$);
	$$\zeta:=8(2g-1)g_1(\ker\mu_a)+4g+8gg_1(\ker\mu)+4g_1(\ker\rho)+(2g+2^{2g+1}(g-1)+3)+14;$$
	 and $$\xi:=34g+28+2(2g+2^{2g+1}(g-1)+3)+ g_3(\ker\rho)+ (8g+1)g_3(\ker\mu)+(32g-15)g_3(\ker\mu_a).$$
\end{notation}

\begin{example} 
If $X$ is an $2$-dimensional iterated Kodaira fibration with injective monodromy and fibers of genus $g$. Then the finite covering $X_2\to X$ constructed by Llosa Isenrich and Py has degree at most $$2^{3}\cdot  \theta(g).$$ 
\end{example}
\begin{proof}
By Theorem \ref{prop}, the degree is at most $$ \theta(g)\cdot \kappa_1(1)\cdot \kappa_2(1)\cdot \kappa_3(1).$$ By Lemma \ref{le1} and $$\kappa_1(1)=I_{1}((\mathbb Z/2\mathbb Z)^{2g},Z_1))=e((\mathbb Z/2\mathbb Z)^{2g})=2,$$ 
$$\kappa_2(1)=I_{1}( (\mathbb Z/2\mathbb Z)^{\tau_0(g)},Z_2)=e(\mathbb Z/2\mathbb Z)^{\tau_0(g)})=2$$ and 
$$\kappa_3(1)=I_{1}((\mathbb Z/2\mathbb Z)^{4g-2},Z_3)=e((\mathbb Z/2\mathbb Z)^{4g-2})=2,$$ the degree from the statement follows.
\end{proof}

Recall $\ker\mu$ is a finite index subgroup of $\pi_1(Y)$ and $\pi_1(Y)$ is a polysurface group of length $n-1$. Consider the filtration of $\pi_1(Y)$: $$1=\mathcal{Y}_0<\mathcal{Y}_1<...<\mathcal{Y}_{n-1}=\pi_1(Y),$$ which gives the filtration of $\ker\mu$: $$1=\mathcal{Y}_0<\mathcal{Y}_1\cap \ker\mu<...<\mathcal{Y}_{n-1}\cap \ker\mu=\ker\mu.$$ We define $g_{1}(\ker\mu)$ to be the genus of the surface corresponding to the surface group $\mathcal{Y}_1\cap \ker\mu$. Similarly, recall $\ker\rho$ is a finite index subgroup of $\pi_1(Y')$ and $\pi_1(Y')$ is a polysurface group of length $n-1$ and there is a filtration of $\pi_1(Y')$: $$1=\mathcal{Y}_0'<\mathcal{Y}_1'<...<\mathcal{Y}_{n-1}'=\pi_1(Y').$$ We define $g_{1}(\ker\rho)$ to be the genus of the surface corresponding to the surface group $\mathcal{Y}_1\cap \ker\rho$.

\begin{example}  
If $X$ is an $3$-dimensional iterated Kodaira fibration with injective monodromy and fibers of genus $g$. Then the finite covering $X_2\to X$ has degree at most $$2^{\zeta}\cdot \theta(g)\cdot \tau(g)^{4g_{1}(\ker\rho)}.$$
\end{example} 
\begin{proof}
By Theorem \ref{prop}, the degree is at most $$\theta(g) \cdot \kappa_1(2)\cdot \kappa_2(2)\cdot  \kappa_3(2).$$ By Remark \ref{a}, $|H_1(R,\mathbb Z/2\mathbb Z)|=2^{2g}$ and $$e(H_1(R',\mathbb Z/2\mathbb Z))=e(H_1(R,\mathbb Z/2\mathbb Z))=e(H_1(L,\mathbb Z/2\mathbb Z))=2,$$ it follows that

\begin{align}
	\kappa_1(2)&=I_2((\mathbb Z/2\mathbb Z)^{2g},Z_1)\nonumber\\
	&=2^5|(\mathbb Z/2\mathbb Z)^{2g}|^{2\cdot 2 g_1(\ker\mu_a)}2^{2g_1(\ker\mu_a)+|r((\mathbb Z/2\mathbb Z)^{2g})|}\nonumber\\
	&=2^{8(2g-1)g_1(\ker\mu_a)+2g_1(\ker\mu_a)+4g+3};\nonumber
\end{align}

\begin{align}
	\kappa_2(2)&=I_2( (\mathbb Z/2\mathbb Z)^{\tau_0(g)},Z_2)\nonumber\\
	&=2^5|(\mathbb Z/2\mathbb Z)^{\tau_0(g)}|^{2\cdot 2 g_1(\ker\rho)}2^{2g_1(\ker\rho)+|r((\mathbb Z/2\mathbb Z)^{\tau_0(g)})|}\nonumber\\
	&=2^{2g_1(\ker\rho)+(2g+2^{2g+1}(g-1)+3)+5}\cdot \tau(g)^{4g_{1}(\ker\rho)} .\nonumber
\end{align}	

\begin{align}
	 \kappa_3(2)&=I_2((\mathbb Z/2\mathbb Z)^{4g-2},Z_3)\nonumber\\
	&=2^5|(\mathbb Z/2\mathbb Z)^{4g-2}|^{2\cdot 2 g_1(\ker\mu)}2^{2g_1(\ker\mu)+|r((\mathbb Z/2\mathbb Z)^{4g-2})|}\nonumber\\
	&=2^{8gg_1(\ker\mu)+2g_1(\ker\mu)+2g+5};\nonumber
\end{align}

Therefore, the degree is at most $$2^\zeta\cdot \theta(g) \cdot  \tau(g)^{4g_{1}(\ker\rho)}.$$
 \end{proof}


\begin{example}  
If $X$ is an $4$-dimensional iterated Kodaira fibration with injective monodromy and fibers of genus $g$. Then the finite covering $X_2\to X$ constructed by Llosa Isenrich and Py has degree at most $$2^{\xi}\cdot \theta(g) \cdot \tau(g)^{4\cdot g_3(\ker\rho)}.$$

\end{example}

\begin{proof}
By Theorem \ref{prop}, the degree is at most $$\theta(g)\cdot \kappa_1(3)\cdot \kappa_2(3)\cdot  \kappa_3(3).$$ 
By Remark \ref{a}, 
\begin{align}
	\kappa_1(3)&=I_3((\mathbb Z/2\mathbb Z)^{2g},Z_1)\nonumber\\
	&= 2^{13}2^{2(2g-1)\cdot 4\cdot g_3(\ker\mu_a)}\cdot(2!)^{6\cdot 2g(2g-1)+ g_3(\ker\mu_a)}\nonumber\\
	&=2^{8(4g-2)\cdot g_3(\ker\mu_a)+24g+g_3(\ker\mu_a)+1};\nonumber
\end{align}	

\begin{align}

	 \kappa_2(3)&=I_3((\mathbb Z/2\mathbb Z)^{\tau_0(g)},Z_2)\nonumber\\
	      &= 2^{13}\cdot \tau(g)^{2\cdot 2\cdot g_3(\ker\rho)}\cdot(2!)^{2|r(H_1(L,\mathbb Z/2\mathbb Z)|+ g_3(\ker\rho)}\nonumber\\
	&=2^{13+2(2g+2^{2g+1}(g-1)+3)+ g_3(\ker\rho)}\cdot \tau(g)^{4\cdot g_3(\ker\rho)};\nonumber
\end{align}	

\begin{align}
	 \kappa_3(3)&=I_3((\mathbb Z/2\mathbb Z)^{4g-2},Z_3)\nonumber\\
	      &= 2^{13}2^{2g\cdot 4\cdot g_3(\ker\mu)}\cdot(2!)^{6\cdot 2g+ g_3(\ker\mu)}\nonumber\\
	&=2^{8g\cdot g_3(\ker\mu)+12g+g_3(\ker\mu)+13}.\nonumber
\end{align}


Therefore, the degree is at most $$2^{\xi}\cdot \theta(g) \cdot  \tau(g)^{4\cdot g_3(\ker\rho)}.$$ 

\end{proof}


\clearpage

\bibliography{refs} 
\bibliographystyle{alpha} 

\end{document}